\documentclass[reqno,12pt]{amsart}
\usepackage[latin1]{inputenc}
\usepackage{enumerate}
\usepackage{amstext,amsmath,amsthm,amsfonts,amssymb,amscd}
\usepackage{latexsym,mathrsfs,euscript}
\usepackage{mathtools}
\usepackage{esint} 
\usepackage{fullpage}
\usepackage{subfig}
\usepackage{hyperref}
\hypersetup{
    colorlinks,%
    citecolor=blue,%
    filecolor=black,%
    linkcolor=red,%
    urlcolor=blue
}

\usepackage{xcolor}
\usepackage{graphicx}

\newtheorem{theorem}{Theorem}

\newtheorem{lemma}[theorem]{Lemma}
\newtheorem{corollary}[theorem]{Corollary}
\theoremstyle{definition}

\theoremstyle{remark}

\numberwithin{equation}{section}
\allowdisplaybreaks
\newcommand{\abs}[1]{\left\vert#1\right\vert}
\newcommand{\set}[1]{\left\{#1\right\}}
\newcommand{\proin}[2]{\left<#1,#2\right>}
\newcommand{\norm}[1]{\left\Vert#1\right\Vert}

\newcommand{\R}{\mathbb{R}}

\newcommand{\D}{\mathscr{D}}
\newcommand{\Ha}{\mathscr{H}}
\newcommand{\F}{\mathscr{F}}

\newcommand{\Ds}{D^\sigma}
\newcommand{\ep}{\varepsilon}

\begin{document}
\title{Nonlocal Schr{\"{o}}dinger equations in metric measure spaces}
\author[M. Actis]{Marcelo Actis}
%
\author[H. Aimar]{Hugo Aimar}

\author[B. Bongioanni]{Bruno Bongioanni}

\author[I. G\'omez]{Ivana G\'{o}mez}

%
%


\thanks{The authors was supported  by CONICET, UNL and ANPCyT (MINCyT)}

\begin{abstract}
 In this note we consider the pointwise convergence to the initial data for the
 solutions of some nonlocal dyadic Schrödinger equations on spaces of
 homogeneous type. We prove the a.e. convergence when the initial data belongs
 to a dyadic version of an $L^2$ based Besov space.
\end{abstract}

\maketitle
%

\section{Introduction}\label{sec:intro}

In quantum mechanics the position of a particle in the \textit{space}
is described by the probability density function
$\abs{\varphi}^2=\varphi \overline{\varphi}$
where $\varphi$ is a solution of the \textit{Schr\"{o}dinger equation}.
In the classical free particle model, the \textit{space} is the euclidean and the
\textit{Schr\"{o}dinger equation} is the associated to the Laplace operator,
i.e.\ $i\tfrac{\partial \varphi}{\partial t}=\triangle\varphi$. Hence the
probability of finding the particle inside the Borel set $E$ of the euclidean
space at time $t$ is given by $\int_E \abs{\varphi(x,t)}^2 dx$.

The pointwise convergence to the initial data for the classical Schrödinger
equation in euclidean settings is a hard problem. It is well known that some
regularity in the initial data is needed
\cite{Carleson80,DaKe82,KeRui83,Cow83,Sjo87,Vega88,TaVa2000}.

Nonlocal operators instead of the Laplacian in this basic model have been
considered previously in the euclidean space (see for example \cite{Las12}
and references in \cite{Val09}). The nonlocal fractional derivatives as
substitutes of the Laplacian become natural objects when the space itself
lacks any differentiable structure and only
an analysis of order less than one can be carried out.

We shall be brief in our introduction of the basic setting. For a more detailed
approach see \cite{ActisAimar}.
Let $(X,d,\mu)$ be a space of homogeneous type (see \cite{MaciasSegovia79}).
Let $\D$ be a dyadic family in $X$ as constructed by M.~Christ in \cite{Christ90}. Let
$\mathscr{H}$ be a Haar system for $L^2(X,\mu)$ associated to $\D$ as built in
\cite{Aimar02} (see also \cite{AiGo00}, \cite{AiBeIa07}). Following the basic
notation introduced in Section~\ref{sec:intro}, we shall use $\mathscr{H}$ itself
s the index set for the analysis and synthesis of signals defined on $X$. Precisely,
by $Q(h)$ we denote the member of $\D$ on which $h$ \textit{is based}. With $j(h)$
we denote the integer scale $j$ for which $Q(h)\in\D^j$.

The system $\mathscr{H}$ is an orthonormal basis for $L^2_0$, where $L^2_0$
coincides with $L^2(X.\mu)$ if $\mu(X)=+\infty$ and
$L^2_0=\{f\in L^2: \int_X f d\mu=0\}$ if $\mu(X)<\infty$. For a given $Q\in\D$
the number of wavelets $h$ based on $Q$ is $\#\vartheta (Q)-1$, where
$\vartheta (Q)$ is the offspring of $Q$ and $\#\vartheta (Q)$ is its cardinal.
The homogeneity property of the space together with the metric control of the
dyadic sets guarantee a uniform upper bound for $\#\vartheta (Q)$. On the other
hand $\#\vartheta (Q)\geq1$ for every $Q\in \D$.

Let $(X,d,\mu,\D,\mathscr{H})$ be given as before. For the sake of simplicity we shall assume along this paper that $X$ itself is a quadrant for $\D$. We say that X itself is a quadrant if any two cubes in $X$ have a common ancestor.
A distance in $X$ associated to $\D$ can be defined by
$\delta(x,y)=\min\{\mu(Q): Q\in\D \text{ such that } x,y \in Q\}$ when $x\neq y$ and $\delta(x,x)=0$.
The next lemma, borrowed to~\cite{ActisAimar}, reflects the one dimensional
character of $X$ equipped with $\delta$ and $\mu$.

\begin{lemma}[Lemma 3.1 in~\cite{ActisAimar}]\label{lem:integrals_of_delta}
 Let $0<\ep<1$, and let $Q$ be a given dyadic cube in $X$.
 Then, for $x\in Q$, we have
 \[\int_{X \backslash Q} \frac{d\mu(y)}{\delta(x,y)^{1+\ep}} \simeq \mu(Q)^{-\ep}.\]
\end{lemma}

Furthermore the integral of $\delta^{-1}(x,\cdot)$ diverges on each dyadic cube
containing $x$ and, when the measure of $X$ is not finite, on the complement of each
dyadic cube.

For a complex value function $f$ Lipschitz continuous with respect to $\delta$ define
\begin{equation*}
D^{\beta}f(x) =\int_{X}\frac{f(x)-f(y)}{\delta(x,y)^{1+\beta}} d\mu(y).
\end{equation*}
One of the key results relating the operator $D^\beta$ with the Haar system
is provided by the following spectral theorem contained in \cite{ActisAimar}.

\begin{theorem}[Theorem 3.1 in \cite{ActisAimar}]\label{thm:eingefunctions_of_Dbeta}
 Let $0<\beta<1$. For each $h \in \Ha$ we have
\begin{equation}\label{eq:eigenfunctions_of_Dsigma}
 D^\beta h(x)=m_h\mu(Q(h))^{-\beta} h(x),
\end{equation}
where $m_h$ is a constant that may depends of $Q(h)$ but there exist two finite and positive constants $M_1$ and $M_2$ such that
\begin{equation}\label{eq:cotas_m_Q}
M_1<m_h<M_2, \quad \text{ for all } h \in \Ha.
\end{equation}
\end{theorem}

Set $B^{\lambda}_{2}(X,\delta,\mu)$ to denote the space of those functions
$f\in L^2(X,\mu)$ satisfying
\begin{equation*}
\iint_{X\times X}\frac{\abs{f(x)-f(y)}^2}{\delta^{1+2\lambda}(x,y)} d\mu(x) d\mu(y)<\infty.
\end{equation*}
The projection operator defined on $L^2$ onto $V_0$ the subspace of functions
which are constant on each cube $Q\in\D^{0}$ is denoted by $P_0$. We are now in
position to state the main results of this paper.

\begin{theorem}\label{thm:Theorem2Main}
Let $0<\beta<\lambda<1$ and $u_0\in B^{\lambda}_2(X, \delta,\mu)$ with $P_0u_0=0$
be given. Then, the function $u$ defined on $\R^+$ by
\begin{equation}\label{eq:seriessolucion}
u(t)=-\sum_{h\in\mathscr{H}}e^{-i t m_h\mu(Q(h))^{-\beta}}\proin{u_0}{h}h,
\end{equation}
\begin{enumerate}[(\ref{thm:Theorem2Main}.a)]
\item belongs to $B^{\lambda}_{2}(X,\delta,\mu)$ for every $t>0$; \label{item:mainresultuno}
\item solves the problem
\begin{equation}\label{eq:Schrodinger_problem}
\begin{dcases}
i \frac{du}{dt} = D^{\beta}u  & t>0,\\
\medskip
u(0)=u_0  & \text{on } X.
\end{dcases}
\end{equation}
More precisely, $\frac{du}{dt}$ is the Fréchet derivative of $u(t)$ as a function of $t\in (0,\infty)$ with values in $B^{\lambda-\beta}_{2}(X,\delta,\mu)$  and $\lim_{t\to0^+}u(t)=u_0$ in
  $B^{\lambda}_{2}(X,\delta,\mu)$.\label{item:mainresultdos}
\end{enumerate}
\end{theorem}

\begin{theorem}\label{thm:Theorem3Main}
Let $0<\beta<\lambda<1$ and $u_0\in B^{\lambda}_2(X, \delta,\mu)$ with $P_0u_0=0$
be given. Then,
\begin{enumerate}[(\ref{thm:Theorem3Main}.a)]
\item there exists $Z\subset X$ with  $\mu(Z)=0$ such that the series \eqref{eq:seriessolucion} defining $u(t)$ converges pointwise for every $t\in [0,1)$ outside $Z$;\label{item:mainresulttres}
\item $u(t)\to u_0$  pointwise almost everywhere on $X$ with respect to $\mu$
when $t\to 0$.\label{item:mainresultcuatro}
\end{enumerate}
\end{theorem}

Let us point out that the operator $D^\beta$ in the framework of $p$-adic analysis is called the Vladimirov operator and has been extensively studied in  recent years, see e.g., \cite{Koz02} and references therein. It has already been proved that the collection of the eigenfunctions of $D^\beta$ coincides with a wavelet basis. Let us also notice that from a probabilistic approach, in \cite{Kig10, Kig13},  the Haar wavelets have been identified as eigenfunctions
of differential operators. On the other hand, since the work of S. Petermichl in \cite{Pet00}, where a representation for the Hilbert transform is given as an average of dyadic shifts, dyadic techniques has proved to be fundamental and useful in harmonic analysis (see \cite{LMP14} and references therein). From
this point of view, our problem can be regarded as a first approximation to the general nonlocal Schrödinger model in Ahlfors metric measure spaces, where
instead of $\delta$ the metric is the underlying distance $d$ in the space.

The paper is organized in three sections. The second one is devoted to introduce a characterization in terms of Haar coefficients
of the dyadic Besov spaces $B^\lambda_2(X,\delta,\mu)$.
In Section~\ref{sec:4} we prove our main results, which are contained in Theorems~\ref{thm:Theorem2Main} and~\ref{thm:Theorem3Main}.
In this last section we also illustrate our results in the Sierpinski gasket.

\section{Characterization of the Besov space
 $B^{\sigma}_2(X,\delta,\mu)$ in terms of Haar coefficients}
 \label{sec:3}

The aim of this section is to characterize $B^{\sigma}_2(X,\delta,\mu)$ in terms of the sequence $\{\proin{f}{h}: h\in\mathscr{H}\}$ for $0<\sigma<1.$

\begin{theorem}\label{thm:characterizationBesov}
let $0<\sigma<1$ be given. The space $B^{\sigma}_2(X,\delta,\mu)$ coincides with
the subspace of $L^2(X,\mu)$ of those functions $f$ for which
\begin{equation*}
\sum_{h\in\mathscr{H}}\frac{\abs{\proin{f}{h}}^2}{\mu(Q(h))^{2\sigma}}<\infty.
\end{equation*}
Moreover,
\begin{equation*}
\norm{f}^2_{B^{\sigma}_2(X,\delta,\mu)}\simeq \norm{f}^2_{L^2(X,\mu)}+ \sum_{h\in\mathscr{H}}\frac{\abs{\proin{f}{h}}^2}{\mu(Q(h))^{2\sigma}}.
\end{equation*}
\end{theorem}

Let us state some important lemmas that will be useful in for the proof of
the above theorem.

\begin{lemma}\label{lem:equivalenceformintegrals}
Let $0<\sigma<1$ and $h,\widetilde h \in \Ha$ be given. Then
\[\nu(h,\widetilde{h}):=\iint_{X\times X} \frac{[h(x)-h(y)][\widetilde{h}(x)-\widetilde{h}(y)]}{\delta(x,y)^{1+2\sigma}}d\mu(y) d\mu(x) = 0,\]
if $h\neq \widetilde h$, and
\[\nu(h,h) =\iint_{X\times X} \frac{[h(x)-h(y)]^2}{\delta(x,y)^{1+2\sigma}}d\mu(y) d\mu(x) \simeq \mu(Q(h))^{-2\sigma},\]
where the equivalence constants depend only on the geometric constants of the space.
\end{lemma}

\begin{proof}
For the first part of the proof, i.e. $\nu(h,\widetilde h)=0$, when
$h\neq \widetilde h$, we shall divide our analysis in three cases according to
the relative positions of $Q(h):=Q$ and $Q(\widetilde{h}):=\widetilde{Q}$,
(i)~$Q=\widetilde{Q}$, (ii) $Q\cap \widetilde{Q}=\emptyset$ and (iii)
$Q\subsetneqq \widetilde{Q}$.

Let us start by (i). Set $\Pi_{h \widetilde{h}}
(x,y) :=[h(x)-h(y)][\widetilde{h}(x)-\widetilde{h}(y)]$. Notice that for $x$ and $y$
in $X\setminus Q$ we have $\Pi_{h \widetilde{h}} (x,y)=0$. On the other hand, for
$x\in Q$ and $y\in X\setminus Q$ we have
$\Pi_{h \widetilde{h}}(x,y)=h(x)\widetilde{h}(x)$. Hence
\begin{align*}
\iint_{X\times X}\frac{\Pi_{h \widetilde{h}} (x,y)}{\delta(x,y)^{1+2\sigma}}d\mu(x) d\mu(y)
=&\iint_{Q\times (X\setminus Q)}\frac{h(x) \widetilde{h}(x)}{\delta(x,y)^{1+2\sigma}}d\mu(x) d\mu(y)\\
& + \iint_{(X\setminus Q)\times Q}\frac{h(y)\widetilde{h}(y)}{\delta(x,y)^{1+2\sigma}}d\mu(x) d\mu(y)\\
& + \iint_{Q\times Q}\frac{\Pi_{h \widetilde{h}} (x,y)}{\delta(x,y)^{1+2\sigma}}d\mu(x) d\mu(y)\\
=& \iint_{Q\times Q}\frac{\Pi_{h \widetilde{h}} (x,y)}{\delta(x,y)^{1+2\sigma}}d\mu(x) d\mu(y),
\end{align*}
since for the first term $\delta(x,y)$ is constant as a function of $x\in Q$ for $y$ fixed in $X\setminus Q$, for the second $\delta(x,y)$ is constant as a function of $y\in Q$ for $x$ fixed in $X\setminus Q$ and $h$ and $\widetilde{h}$ are orthogonal. Let us prove that $\iint_{Q\times Q}\frac{\Pi_{h \widetilde{h}} (x,y)}{\delta(x,y)^{1+2\sigma}}d\mu(x) d\mu(y)=0$.
Since $h$ and $\widetilde{h}$ are constant on $Q'\in\vartheta(Q)$ we have that $\Pi_{h \widetilde{h}}(x,y)=0$ for $(x,y)\in Q'\times Q'$, hence
\begin{align}\label{eq:Pionequalcube}
\iint_{Q\times Q}\frac{\Pi_{h \widetilde{h}} (x,y)}{\delta(x,y)^{1+2\sigma}}d\mu(x) d\mu(y)\notag&\\
&\hspace{-3truecm}=\sum_{Q'\in \vartheta(Q)}\sum_{Q''\in \vartheta(Q)}\int_{Q''}\int_{Q'}\frac{\Pi_{h \widetilde{h}} (x,y)}{\delta(x,y)^{1+2\sigma}} d\mu(x) d\mu(y)\notag\\
&\hspace{-3truecm}= \sum_{Q'\in \vartheta(Q)}\int_{Q'}\int_{Q'}\frac{\Pi_{h \widetilde{h}} (x,y)}{\delta(x,y)^{1+2\sigma}}d\mu(x) d\mu(y) \notag\\
 &  +  \sum_{Q'\in \vartheta(Q)}\sum_{\substack{Q''\in \vartheta(Q)\\Q'\neq Q''}}\int_{Q''}\int_{Q'}\frac{\Pi_{h \widetilde{h}} (x,y)}{\mu(Q)^{1+2\sigma}}d\mu(x) d\mu(y)\notag\\
&\hspace{-3truecm}= \mu(Q)^{-1-2\sigma}\sum_{Q''\in\vartheta(Q)}\int_{Q''}\left(\sum_{Q'\in\vartheta(Q)}\int_{Q'}\Pi_{h \widetilde{h}} (x,y)d\mu(x)\right) d\mu(y)\notag\\
&\hspace{-3truecm}= \mu(Q)^{-1-2\sigma}\int_{Q} \int_{Q} [h(x)\widetilde{h}(x)-h(y)\widetilde{h}(x) \notag\\
& \hspace*{1truecm}-h(x)\widetilde{h}(y)+h(y)\widetilde{h}(y)] d\mu(x)d\mu(y)\\
&\hspace{-3truecm}= 0. \notag
\end{align}

Let us now consider the case (ii), that is $Q\cap \widetilde{Q}=\emptyset$. In this case $\Pi_{h \widetilde{h}}(x,y)$ is supported in $(\widetilde{Q}\times Q)\cup(Q\times \widetilde{Q})$. Moreover on $\widetilde{Q}\times Q$ we have $\Pi_{h \widetilde{h}}(x,y)=-h(y) \widetilde{h}(x)$ and on $Q\times \widetilde{Q}$, $\Pi_{h \widetilde{h}}(x,y)=-h(x) \widetilde{h}(y)$. Since, on the other hand $\delta(x,y)=\delta(Q,\widetilde{Q})$ which is a positive constant on the support of $\Pi_{h \widetilde{h}}$, we get
\begin{align*}
&\iint_{X\times X}\frac{\Pi_{h \widetilde{h}} (x,y)}{\delta(x,y)^{1+2\sigma}}d\mu(x) d\mu(y)\\
&= -\frac{1}{\delta(Q,\widetilde{Q})^{1+2\sigma}}\left\{\iint_{\widetilde{Q}\times Q}[h(y)\widetilde{h}(x)] d\mu(x)d\mu(y)
+ \iint_{Q\times \widetilde{Q}}[h(x)\widetilde{h}(y)] d\mu(x)d\mu(y)\right\}\\
&=0.
\end{align*}
Consider now the case (iii). Since $Q\varsubsetneqq \widetilde{Q}$, then $\widetilde{h}$ is constant on $Q$, hence
\begin{equation*}
\iint_{X\times X}\frac{\Pi_{h \widetilde{h}} (x,y)}{\delta(x,y)^{1+2\sigma}}d\mu(x) d\mu(y) = \iint_{(X\setminus Q)\times Q} + \iint_{Q\times Q} + \iint_{Q\times (X\setminus Q)}.
\end{equation*}
Since $\widetilde{h}$ is constant on $Q$, $\Pi_{h \widetilde{h}}$ is identically zero on $Q\times Q$ and the second term vanishes. For the first term notice that it can be written as
\begin{align*}
\int_{X\setminus Q}&\left(\int_{Q}\frac{(-h(y))(\widetilde{h}(x)-\widetilde{h}(y_Q))}{\delta(Q,x)^{1+2\sigma}}d\mu(y)\right) d\mu(x)\\
&=-\int_{X\setminus Q}\frac{\widetilde{h}(x)-\widetilde{h}(y_Q)}{\delta(Q,x)^{1+2\sigma}}d\mu(x)\left(\int_Q h(y) d\mu(y)\right)\\
&=0,
\end{align*}
here $y_Q$ denotes any fixed point in $Q$. For the third term we have similarly
\begin{align*}
\iint_{Q\times (X\setminus Q)}&\frac{\Pi_{h \widetilde{h}} (x,y)}{\delta(x,y)^{1+2\sigma}}d\mu(x) d\mu(y)\\
&= \int_{X\setminus Q}\left(\int_{Q} \frac{(h(x))(\widetilde{h}(x_Q)-\widetilde{h}(y))}{\delta(Q,y)^{1+2\sigma}}d\mu(x)\right) d\mu(y) \\
&=0.
\end{align*}

Finally we have to show that $\nu(h,h)\simeq \mu(Q(h))^{-2\sigma}$. Let $Q=Q(h)$. Notice first for $(x,y)\in (X\setminus Q)\times (X\setminus Q)$ we have
$\Pi_{h h} (x,y)=0$. Hence
\begin{align*}
\iint_{X\times X}&\frac{\Pi_{h h} (x,y)}{\delta(x,y)^{1+2\sigma}}d\mu(x) d\mu(y)\\
&= \iint_{(X\setminus Q)\times Q}\frac{\Pi_{h h} (x,y)}{\delta(x,y)^{1+2\sigma}}d\mu(x) d\mu(y) +
\iint_{Q\times Q}\frac{\Pi_{h h} (x,y)}{\delta(x,y)^{1+2\sigma}}d\mu(x) d\mu(y)\\
& \phantom{\iint_{(X\setminus Q)\times Q}\frac{\Pi_{h h} (x,y)}{\delta(x,y)^{1+2\sigma}}}+ \iint_{Q\times (X\setminus Q)}\frac{\Pi_{h h} (x,y)}{\delta(x,y)^{1+2\sigma}}d\mu(x) d\mu(y)\\
&= 2\int_{X\setminus Q}\left(\int_{Q} \frac{\Pi_{h h} (x,y)}{\delta(x,y)^{1+2\sigma}}d\mu(x)\right)d\mu(y) + \iint_{Q\times Q}\frac{\Pi_{h h} (x,y)}{\delta(x,y)^{1+2\sigma}}d\mu(x) d\mu(y)\\
&= 2I + II.
\end{align*}
Let us first get an estimate for $I$. Notice that for any $x,z\in Q$ and $y\in X\setminus Q$, we have that $\delta(x,y)=\delta(z,y)$, hence $I=\int_{X\setminus Q}\delta(z,y)^{-(1+2\sigma)}d\mu(y)\int_Q\abs{h(x)}^2 d\mu(x)$, which is equivalent to $\mu(Q)^{-2\sigma}$ by Lemma~\ref{lem:integrals_of_delta}. To get the desired bound for $II$, we observe that equation \eqref{eq:Pionequalcube} holds for $h=\widetilde{h}$ also, then
\begin{align*}
\iint_{Q\times Q}\frac{\Pi_{h h} (x,y)}{\delta(x,y)^{1+2\sigma}}& d\mu(x) d\mu(y)\\
&=\mu(Q)^{-1-2\sigma}\iint_{Q\times Q}[h^2(x)+ h^2(y)-2h(x)h(y)] d\mu(x)d\mu(y)\\
&= 2\mu(Q)^{-2\sigma}.
\end{align*}
\end{proof}

For $Q\in \D$, set $\Ha_Q=\{h\in\Ha: Q(h)\subseteq Q\}$. Let $\mathcal{S}(\Ha_Q)$
denote the linear span of $\Ha_Q$. Since $\Ha_Q$ is countable we write $\ell^2$ to
denote the space $\ell^2(\Ha_Q)$ of all square summable complex sequences indexed
on $\Ha_Q$. On the other hand, consider the weighted space
$\mathcal{L}^2(Q\times Q):=L^2(Q\times Q, \frac{d\mu(x)d\mu(y)}{\delta(x,y)})$.
The next lemma shows that for $\varphi$ and $\psi$ in $\mathcal{S}(\Ha_Q)$
the inner product of $\frac{\varphi(x)-\varphi(y)}{\delta(x,y)^\sigma}$ with
$\frac{\psi(x)-\psi(y)}{\delta(x,y)^\sigma}$ in $\mathcal{L}^2(Q\times Q)$ is
equivalent to the inner product of $\frac{\proin{\varphi}{h}}{\mu(Q(h))^\sigma}$
with $\frac{\proin{\psi}{h}}{\mu(Q(h))^\sigma}$ in $\ell^2(\Ha_Q)$.

\begin{lemma}\label{lem:polarization}
Let $0<\sigma<1$ and $Q\in\D$ be given. For $\varphi$, $\psi$ two functions in $\mathcal{S}(\mathscr{H}_Q)$ we have that
\begin{equation*}
\iint_{Q\times Q}\frac{[\varphi(x)-\varphi(y)][\psi(x)-\psi(y)]}{\delta(x,y)^{1+2\sigma}} d\mu(x) d\mu(y)
=\sum_{h\in \mathscr{H}_Q}\proin{\varphi}{h}\proin{\psi}{h}\nu(h,h).
\end{equation*}
In particular,
\begin{equation*}
\iint_{Q\times Q}\frac{\abs{\varphi(x)-\varphi(y)}^2}{\delta(x,y)^{1+2\sigma}} d\mu(x) d\mu(y)\simeq\sum_{h\in \mathscr{H}_Q}\frac{\abs{\proin{\varphi}{h}}^2}{\mu(Q(h))^{2\sigma}}.
\end{equation*}
\end{lemma}
\begin{proof}
Since $\varphi$ and $\psi$ are in the linear span of $\mathscr{H}_Q$ we easily see that
\begin{equation*}
[\varphi(x)-\varphi(y)][\psi(x)-\psi(y)]= \sum_{h\in \mathscr{H}_Q}\sum_{\widetilde{h}\in\mathscr{H}_Q}\proin{\varphi}{h}\proin{\psi}{\widetilde{h}}[h(x)-h(y)][\widetilde{h}(x)-\widetilde{h}(y)]
\end{equation*}
Dividing both members of the above equation by $\delta(x,y)^{1+2\sigma}$, integrating on the  product space $Q\times Q$ and then applying Lemma \ref{lem:equivalenceformintegrals} we get
\begin{align*}
\iint_{Q\times Q}&\frac{[\varphi(x)-\varphi(y)][\psi(x)-\psi(y)]}{\delta(x,y)^{1+2\sigma}} d\mu(x) d\mu(y)\\
&=\sum_{h\in \mathscr{H}_Q}\sum_{\widetilde{h}\in\mathscr{H}_Q}\proin{\varphi}{h}\proin{\psi}{\widetilde{h}}\iint_{Q\times Q}\frac{(h(x)-h(y))(\widetilde{h}(x)-\widetilde{h}(y))}{\delta(x,y)^{1+2\sigma}} d\mu(x) d\mu(y) \\
&=\sum_{h\in \mathscr{H}_Q}\proin{\varphi}{h}\proin{\psi}{h}\iint_{Q\times Q}\frac{[h(x)-h(y)]^2}{\delta(x,y)^{1+2\sigma}} d\mu(x) d\mu(y)\\
&=\sum_{h\in \mathscr{H}_Q}\proin{\varphi}{h}\proin{\psi}{h}\nu(h,h).
\end{align*}
\end{proof}

\begin{lemma}\label{lem:diferenciaentornodiagonal}
For $\psi\in\mathcal{S}(\mathscr{H})$ there exists $\ep>0$ such that $\psi(x)-\psi(y)$ vanishes on $\Delta_{\epsilon}=\set{(x,y)\in X\times X: \delta(x,y)<\ep}$.
\end{lemma}
\begin{proof}
It is enough to check the result for $\psi=h\in\mathscr{H}$. But since $h$ is
constant on each child $Q'$ of $Q(h)$ we have that $h(x)-h(y)=0$ on
$\bigcup_{Q' \textrm{ child of }Q} Q'\times Q'$. On the other hand,
$h(x)-h(y)=0$ for $x$ and $y$ both outside $Q(h)$. Hence $h(x)-h(y)$ vanishes
on $\set{(x,y): \delta(x,y)<\mu(Q(h))}$.
\end{proof}

The next result shows that Lemma~\ref{lem:polarization} extends to the case of $\varphi$ in $L^2$.
\begin{lemma}\label{lem:polarizationextended}
Let $Q$ be a cube in $\D$, $f\in L^2(Q,\mu)$ and $\psi\in \mathcal{S}(\mathscr{H}_Q)$. Then
\begin{equation*}
\iint_{Q\times Q}\frac{(f(x)-f(y))(\psi(x)-\psi(y))}{\delta(x,y)^{1+2\sigma}} d\mu(x) d\mu(y)
=\sum_{h\in \mathscr{H}_Q}\proin{f}{h}\proin{\psi}{h}\nu(h,h).
\end{equation*}
\end{lemma}
\begin{proof}
Let $\mathscr{F}_n\nearrow\mathscr{H}_Q$ and $f_n=\sum_{h\in\mathscr{F}_n}\proin{f}{h}h$. Since each $f_n\in\mathcal{S}(\mathscr{H}_Q)$ we have from Lemma~\ref{lem:polarization} that
\begin{equation*}
\iint_{Q\times Q}\frac{[f_n(x)-f_n(y)][\psi(x)-\psi(y)]}{\delta(x,y)^{1+2\sigma}} d\mu(x) d\mu(y)
=\sum_{h\in \mathscr{H}_Q}\proin{f}{h}\proin{\psi}{h}\nu(h,h)
\end{equation*}
taking $n$ large enough such that $\mathscr{F}_n$ contains all the $h$'s building $\psi$. On the other hand, since, from Lemma~\ref{lem:diferenciaentornodiagonal}, the function $\frac{\psi(x)-\psi(y)}{\delta(x,y)^{1+2\sigma}}\in L^{\infty}(Q\times Q)$ and $f_n(x)-f_n(y)\to f(x)-f(y)$ in $L^2(Q\times Q)$, hence in $L^1(Q\times Q)$ we get the desired equality.
\end{proof}

The next result allow us to localize to dyadic cubes our characterization of the Besov space $B^{\sigma}_2(X,\delta,\mu)$ in terms of the Haar coefficients.
\begin{lemma}\label{lem:unioncubesmaximos}
The set
\begin{equation*}
\triangle_1=\set{(x,y)\in X\times X: \delta(x,y)<1}=\bigcup_{Q\in\mathscr{M}} Q\times Q,
\end{equation*}
where $\mathscr{M}$ is the family of all maximal dyadic cubes in $\D$ with $\mu(Q)<1$.
\end{lemma}

The next result is a statement of the intuitive fact that the regularity requirement additional to the $L^2$ integrability involved in the definition of $B^{\sigma}_2(X,\delta,\mu)$ is only relevant around the diagonal of $X\times X$.

\begin{lemma}\label{lem:Besovdistancelessone}
For $0<\sigma<1$ an $L^2$ function $f$ belongs to $B^{\sigma}_2(X,\delta,\mu)$ if and only if
\begin{equation*}
\mathscr{E}(f)=\iint_{\delta(x,y)<1}\frac{\abs{f(x)-f(y)}^2}{\delta(x,y)^{1+2\sigma}}d\mu(x)d\mu(y)<\infty.
\end{equation*}
Moreover, the $B^{\sigma}_2(X,\delta,\mu)$ norm of $f$ is equivalent to $\norm{f}_2+[\mathscr{E}(f)]^{\tfrac{1}{2}}$.
\end{lemma}
\begin{proof}
The ``only if'' is obvious. Since
\begin{equation*}
\iint_{\delta(x,y)\geq 1}\frac{\abs{f(x)-f(y)}^2}{\delta(x,y)^{1+2\sigma}} d\mu(x) d\mu(y)
\leq C\int_{x\in X}\abs{f(x)}^2\left(\int_{\delta(x,y)\geq 1}\frac{d\mu(y)}{\delta(x,y)^{1+2\sigma}}\right) d\mu(x).
\end{equation*}
From Lemma~\ref{lem:integrals_of_delta} we have that the first term in the above inequality is bounded by the $L^2$ norm of $f$, as desired.
\end{proof}

The following result is elementary but useful.
\begin{lemma}\label{lem:equivalenceSquarefunction}
For an $L^2(X,\mu)$ function $f$ the quantities $\norm{f}^2_2 + \sum_{h\in\mathscr{H}}\frac{\abs{\proin{f}{h}}^2}{\mu(Q(h))^{2\sigma}}$ and
$\norm{f}^2_2 + \sum_{\substack{h\in\mathscr{H}\\ \mu(Q(h))<1}}\frac{\abs{\proin{f}{h}}^2}{\mu(Q(h))^{2\sigma}}$ are equivalent.
\end{lemma}

The characterization  of $B^{\sigma}_2(X,\delta,\mu)$ in terms of the Haar system is just a completation argument built on the result in Lemma \ref{lem:equivalenceformintegrals}.

\begin{proof}[Proof of Theorem~\ref{thm:characterizationBesov}]
From Lemmas~\ref{lem:diferenciaentornodiagonal}, \ref{lem:polarizationextended} and \ref{lem:unioncubesmaximos} it is enough to prove that that quantities
\begin{equation}\label{eq:integralform}
\iint_{Q\times Q}\frac{\abs{f(x)-f(y)}^2}{\delta(x,y)^{1+2\sigma}} d\mu(x) d\mu(y)
 \end{equation}
 and
\begin{equation}\label{eq:sumform}
\sum_{h\in\mathscr{H}_Q}\frac{\abs{\proin{f}{h}}^2}{\mu(Q(h))^{2\sigma}}
\end{equation}
are equivalent for each cube $Q$ in $\mathscr{M}$ with constants independent of $Q$.

Let us start by showing that the double integral in~\eqref{eq:integralform} is
bounded by the sum~\eqref{eq:sumform}. Let $(\mathscr{F}_n)$ be an increasing
sequence of finite subfamilies of $\mathscr{H}_Q$ that covers $\mathscr{H}_Q$.
In other words $\mathscr{F}_n\subseteq \mathscr{F}_{n+1}$ and
$\mathscr{H}_Q=\bigcup_n \mathscr{F}_n$.
Let $f_n=\sum_{h\in \mathscr{F}_n}\proin{f\chi_Q}{h}h$. Since $f_n$ converge
pointwise almost everywhere for $x$ in $Q$ to $f(x)$, then we have that
$\abs{f_n(x)-f_n(y)}\to \abs{f(x)-f(y)}$ as $n\to \infty$ for $y\in Q$ also.
Hence from Fatou's lemma and Lemma~\ref{lem:polarization}
\begin{align*}
\iint_{Q\times Q}\frac{\abs{f(x)-f(y)}^2}{\delta(x,y)^{1+2\sigma}} d\mu(x) d\mu(y)
&=\iint_{Q\times Q}\lim_{n\to\infty}\frac{\abs{f_n(x)-f_n(y)}^2}{\delta(x,y)^{1+2\sigma}} d\mu(x) d\mu(y)\\
& \leq \liminf_{n\to\infty}\iint_{Q\times Q}\frac{\abs{f_n(x)-f_n(y)}^2}{\delta(x,y)^{1+2\sigma}} d\mu(x) d\mu(y)\\
&\leq C \liminf_{n\to\infty}\sum_{h\in \mathscr{H}_Q}\frac{\abs{\proin{f_n}{h}}^2}{\mu(Q(h))^{2\sigma}} \\
&\leq C \sum_{h\in \mathscr{H}_Q}\frac{\abs{\proin{f}{h}}^2}{\mu(Q(h))^{2\sigma}}.
\end{align*}

To prove the converse we shall argue by duality. Take
$\boldmath{b}= \set{b_h: h\in \mathscr{H}_Q}$ a sequence of scalars indexed
on $\mathscr{H}_Q$ with $b_h=0$ except for a finite number of $h\in\mathscr{H}_Q$.
Assume that $\sum_{h\in\mathscr{H}_Q}\abs{b_h}^2\leq 1$. Set
$\psi= \sum_{h\in\mathscr{H}_Q}b_h\nu_(h,h)^{-1/2}h \in \mathcal{S}(\mathscr{H}_Q)$.
So that from Lemma~\ref{lem:polarizationextended},

\begin{align*}
\sum_{h\in \mathscr{H}_Q}\proin{f}{h}\nu(h,h)^{\tfrac{1}{2}} b_h
= \sum_{h\in \mathscr{H}_Q}\proin{f}{h}\proin{\psi}{h}\nu(h,h)&\\
& \hspace*{-7.5truecm}=\iint_{Q\times Q}\frac{(f(x)-f(y))(\psi(x)-\psi(y))}{\delta(x,y)^{1+2\sigma}} d\mu(x) d\mu(y)\\
& \hspace*{-7.5truecm}\leq \left(\iint_{Q\times Q}\frac{\abs{f(x)-f(y)}^2}{\delta(x,y)^{1+2\sigma}} d\mu(x) d\mu(y)\right)^{\tfrac{1}{2}}
\left(\iint_{Q\times Q}\frac{\abs{\psi(x)-\psi(y)}^2}{\delta(x,y)^{1+2\sigma}} d\mu(x) d\mu(y)\right)^{\tfrac{1}{2}}.
\end{align*}
Notice that from Lemma~\ref{lem:polarization} and Lemma~\ref{lem:equivalenceformintegrals} we have,
\begin{align*}
\iint_{Q\times Q}\frac{\abs{\psi(x)-\psi(y)}^2}{\delta(x,y)^{1+2\sigma}} d\mu(x) d\mu(y)
&\leq C \sum_{h\in \mathscr{H}_Q}\frac{\abs{\proin{\psi}{h}}^2}{\mu(Q(h))^{2\sigma}}\\
&\leq C \sum_{h\in \mathscr{H}_Q} \abs{b_h}^2\left[\frac{\nu(h,h)^{-\tfrac{1}{2}}}{\mu(Q(h))^{\sigma}}\right]^2\\
&\leq C.
\end{align*}
Hence
\begin{equation*}
\abs{\sum_{h\in \mathscr{H}_Q}\proin{f}{h}\nu(h,h)^{\tfrac{1}{2}} b_h}\leq C \left(\iint_{Q\times Q} \frac{\abs{f(x)-f(y)}^2}{\delta(x,y)^{1+2\sigma}} d\mu(x) d\mu(y)\right)^{\tfrac{1}{2}},
\end{equation*}
so that
\begin{equation*}
\left(\sum_{h\in \mathscr{H}_Q}\abs{\proin{f}{h}}^2\nu(h,h)\right)^{\tfrac{1}{2}}\leq C\left(\iint_{Q\times Q} \frac{\abs{f(x)-f(y)}^2}{\delta(x,y)^{1+2\sigma}} d\mu(x) d\mu(y)\right)^{\tfrac{1}{2}}
\end{equation*}
or
\begin{equation*}
\left(\sum_{h\in \mathscr{H}_Q}\frac{\abs{\proin{f}{h}}^2}{\mu(Q(h))^{2\sigma}}\right)^{\tfrac{1}{2}}\leq C\left(\iint_{Q\times Q} \frac{\abs{f(x)-f(y)}^2}{\delta(x,y)^{1+2\sigma}} d\mu(x) d\mu(y)\right)^{\tfrac{1}{2}}
\end{equation*}
as desired.

\end{proof}

\section{Proof of Theorems \ref{thm:Theorem2Main} and \ref{thm:Theorem3Main}}
\label{sec:4}
Before start proving the main results, let us
mention that from the result of the previous section the Sobolev type condition
required in \cite[Theorem 4.1]{ActisAimar} is the same as our current hypothesis
$f\in B^\lambda_2(X,\delta,\mu)$ when $p=2$. Then for any function $f\in
B^\sigma_2(X,\delta,\mu)$ we can write
\[\Ds f(x)=\sum_{h \in \Ha} m_h\mu(Q(h))^{-\sigma} \proin{f}{h} h(x),\]
where the series converges in $L^2(X,\mu)$.

\begin{proof}[Proof of Theorem~\ref{thm:Theorem2Main}]
From Theorem~\ref{thm:characterizationBesov} we see that for each $t>0$, $u(t)$
defined in~\eqref{eq:seriessolucion} belongs to $B^{\lambda}_{2}(X,\delta,\mu)$, since $u_0\in B^{\lambda}_2(X,\delta,\mu)$. Moreover, for $t,s\geq 0$,
\begin{align*}
\norm{u(t)-u(s)}^2_{B^{\lambda}_{2}(X,\delta,\mu)}
&= \norm{\sum_{h\in\mathscr{H}} \left(e^{itm_h \mu({Q(h)})^{-\beta}}-
e^{i sm_h\mu({Q(h)})^{-\beta}}\right)
\proin{u_0}{h}h}^2_{B^{\lambda}_{2}(X,\delta,\mu)}\\
&= \sum_{h\in\mathscr{H}}\abs{e^{i tm_h \mu({Q(h)})^{-\beta}}-
e^{i sm_h \mu({Q(h)})^{-\beta}}}^2\abs{\proin{u_0}{h}}^2 \\
&\qquad+ \sum_{h\in\mathscr{H}} \abs{e^{i tm_h \mu({Q(h)})^{-\beta}}-
e^{i sm_h \mu({Q(h)})^{-\beta}}}^2
\frac{\abs{\proin{u_0}{h}}^2}{\mu({Q(h)})^{2\lambda}}.
\end{align*}
In order to see that both series above tend to zero as $s\to t$, for each
positive $\ep$ we can take a finite subfamily $\F$ of $\Ha$ such that
$\sum_{h\in\Ha\backslash\F} \abs{\proin{u_0}{h}}^2
\left(1+\mu(Q(h))^{-2\lambda}\right)<\ep$. For the sums on $\F$ we can argue
with the continuity of the complex exponential.

Let us prove that the formal derivative of $u(t)$ is actually the derivative in the sense of $B^{\lambda-\beta}_{2}(X,\delta,\mu)$. In fact, for $t>0$ and $\tau$ small
enough,
\begin{align*}
&\norm{\frac{u(t+\tau)-u(t)}{\tau}-
i\sum_{h\in\mathscr{H}}e^{itm_h\mu({Q(h)})^{-\beta}}\mu({Q(h)})^{-\beta}
\proin{u_0}{h}h}^2_{B^{\lambda-\beta}_{2}(X,\delta,\mu)}\\
&=\norm{\sum_{h\in\mathscr{H}}e^{itm_h\mu({Q(h)})^{-\beta}}
\left[\frac{e^{i\tau m_h\mu({Q(h)})^{-\beta}}-1}{\tau}-i\mu({Q(h)})^{-\beta}\right]
\proin{u_0}{h}h}^2_{B^{\lambda-\beta}_{2}(X,\delta,\mu)}\\
&\leq C \left\{\norm{\sum_{h\in\mathscr{H}}e^{itm_h\mu({Q(h)})^{-\beta}}
\left[\frac{e^{i\tau m_h\mu({Q(h)})^{-\beta}}-1}{\tau}-i\mu({Q(h)})^{-\beta}\right]
\proin{u_0}{h}h}^2_{L^2}\right.\\
&\hspace*{3truecm}
\left.+\sum_{h\in\mathscr{H}}\abs{\frac{e^{i\tau m_h\mu({Q(h)})^{-\beta}}-1}{\tau}
-i\mu({Q(h)})^{-\beta}}^2
\frac{\abs{\proin{u_0}{h}}^2}{\mu({Q(h)})^{2(\lambda-\beta)}}\right\}\\
&\leq C \sum_{h\in\mathscr{H}}\mu({Q(h)})^{2\beta}
\abs{\frac{e^{i\tau m_h\mu({Q(h)})^{-\beta}}-1}{\tau}-i\mu({Q(h)})^{-\beta}}^2
\frac{\abs{\proin{u_0}{h}}^2}{\mu({Q(h)})^{2\lambda}}.
\end{align*}
Since, from Theorem~\ref{thm:characterizationBesov},
$\sum_{h\in\mathscr{H}}\frac{\abs{\proin{u_0}{h}}^2}{\mu({Q(h)})^{2\lambda}}<\infty$
and \[\mu({Q(h)})^{2\beta}\left|\frac{e^{i\tau m_h\mu({Q(h)})^{-\beta}}-1}{\tau}
-i\mu({Q(h)})^{-\beta}\right|^2= \left|\frac{e^{i\tau  m_h\mu({Q(h)})^{-\beta}}
-1}{\mu({Q(h)})^{-\beta}\tau}-i\right|^2\]
tends to zero as $\tau\to 0$ for each $h\in\Ha$, arguing as before we obtain
the result.

On the other hand since $u(t)\in B^{\lambda}_{2}(X,\delta,\mu)$ and since
 $\lambda>\beta$, $D^{\beta}u(t)$ is well defined and it is given by
\begin{align*}
D^{\beta}u(t) &= D^{\beta}\left(\sum_{h\in\mathscr{H}}e^{itm_h\mu({Q(h)})^{-\beta}}
\proin{u_0}{h}h\right) \\
& = \sum_{h\in\mathscr{H}}e^{itm_h\mu({Q(h)})^{-\beta}}\mu({Q(h)})^{-\beta}
\proin{u_0}{h}h = -i \frac{d u}{d t}.
\end{align*}
Hence $u(t)$ is a solution of the nonlocal equation and \textit{(\ref{thm:Theorem2Main}.\ref{item:mainresultdos})} is proved.

\end{proof}

\bigskip

Before proving Theorem~\ref{thm:Theorem3Main} we shall obtain some basic maximal estimates involved in the proofs of \textit{(\ref{thm:Theorem3Main}.\ref{item:mainresulttres})} and \textit{(\ref{thm:Theorem3Main}.\ref{item:mainresultcuatro})}. With $M_{dy}$ we denote the Hardy-Littlewood dyadic maximal operator given by
\begin{equation*}
M_{dy}f(x)=\sup \frac{1}{\mu({Q})}\int_{Q}\abs{f(y)}d\mu(y)
\end{equation*}
where the supremum is taken on the family of all dyadic cubes $Q\in\D$ for which $x\in Q$. Calder\'on's dyadic maximal operator of order $\lambda$ is defined by
\begin{equation*}
M^{\#}_{\lambda,dy}f(x)=\sup\frac{1}{\mu({Q})^{1+\lambda}}\int_{Q}\abs{f(y)-f(x)}d\mu(y),
\end{equation*}
where the supremum is taken on the family of all dyadic cubes $Q$ of $X$  such that $x\in Q$. The following lemma can be seen as an extension of a result by DeVore and Sharpley in \cite[Corollary 11.6]{DeSa84}.

\begin{lemma}\label{lem:acotacionMaximalCalderon}
If $f\in B^{\lambda}_{2}(X,\delta,\mu)$, then $\norm{M^{\#}_{\lambda, dy}f}_{L^2}\leq \norm{f}_{B^{\lambda}_{2}}$.
\end{lemma}
\begin{proof}
Let $Q\in\D$ and $x\in Q$ be given.
Applying Schwarz's inequality, since $\delta(x,y)\leq \mu(Q)$ for $y\in Q$, we have
\begin{align*}
\frac{1}{\mu(Q)^{1+\lambda}}\int_{Q}\abs{f(y)-f(x)}d\mu(y) &
\leq \frac{1}{\mu(Q)^{1+\lambda}}\left(\int_{Q}\abs{f(y)-f(x)}^2 d\mu(y)\right)^{\tfrac{1}{2}}\mu(Q)^{\tfrac{1}{2}}\\
& \leq \left(\frac{1}{\mu(Q)^{1+2\lambda}}\int_{Q}\abs{f(y)-f(x)}^2 d\mu(y)\right)^{\tfrac{1}{2}}\\
& \leq \left(\int_{X}\frac{\abs{f(y)-f(x)}^2}{\delta(x,y)^{1+2\lambda}} d\mu(y)\right)^{\tfrac{1}{2}}.
\end{align*}
\end{proof}

In this section we shall need a better description of the structure of $\mathscr{H}$ in terms of scales. As we said before $\D^j$ denotes the dyadic sets at the scale $j\in \mathbb{Z}$. With $\mathscr{H}^j$ we denote the Haar wavelets which are based in the cubes of $\D^j$ for $j$ a fixed integer. Since from our construction of the Haar system we have that
\begin{equation*}
V_{j+1}=V_j\oplus W_j,
\end{equation*}
then $\mathscr{H}^j$ is an orthonormal basis for $W_j$. Hence, for $N>0$, $V_{N+1}=V_0\oplus W_0\oplus \ldots \oplus W_N$. So that, for $P_0 f=0$, the projection of $f$ onto $V_{N+1}$ is given by
\begin{equation*}
P_{N+1}f=\sum_{j=0}^N \mathcal{Q}_j f,
\end{equation*}
where $\mathcal{Q}_j$ is the projector onto $W_j$, precisely
\begin{equation*}
\mathcal{Q}_j f=\sum_{h\in \mathscr{H}^j}\proin{f}{h} h.
\end{equation*}
As it is easy to see, since $V_{N+1}$ is the space of those $L^2$ functions on $X$ which are constant on each cube of $\D^{N+1}$, $\abs{P_{N+1}f(x)}\leq M_{dy}f(x)$, pointwise.

Let us next introduce two maximal operators related to the series \eqref{eq:seriessolucion}. For fixed $t>0$ set
\begin{equation*}
S_{t}^{\displaystyle *}f(x)=\sup_{N\in \mathbb{N}}\abs{S^{N}_{t}f(x)}
\end{equation*}
with
\begin{equation*}
S^{N}_{t}f(x)=\sum_{j=0}^{N}\sum_{h\in\mathscr{H}^j}e^{-i t m_h\mu(Q(h))^{-\beta}}\proin{f}{h}h(x).
\end{equation*}
Set
\begin{equation*}
S^{\displaystyle *} f(x)=\sup_{0<t<1}S_{t}^{\displaystyle *}f(x).
\end{equation*}

\begin{lemma}\label{lem:estimacionesdosmaximales}
Let $f\in B^{\lambda}_2(X,\delta,\mu)$ with $0<\beta<\lambda<1$ and $P_0f=0$. Then, the inequalities
\begin{enumerate}[(\ref{lem:estimacionesdosmaximales}.a)]
\item $S_{t}^{\displaystyle *}f(x)\leq Ct M^{\#}_{\lambda, dy}f(x) + 2 M_{dy}f(x)$
for $t\geq 0$ and $x\in X$;\label{item:estimacionesdosmaximalesuno}

\item $S^{\displaystyle *}f(x)\leq C M^{\#}_{\lambda, dy}f(x) + 2 M_{dy}f(x)$ for $x\in X$;\label{item:estimacionesdosmaximalesdos}

\item $\norm{S^{\displaystyle *}f}_{L^2}\leq C\norm{f}_{B^{\lambda}_2 (X,\delta,\mu)}$,\label{item:estimacionesdosmaximalestres}
\end{enumerate}
hold for some constant $C$ which does not depend on $f$.
\end{lemma}

\begin{proof}
For $f\in B^{\lambda}_2 (X,\delta,\mu)$, $t\geq 0$ and $N\in \mathbb{N}$, we have
\begin{equation}\label{eq:sumaryrestar}
\abs{S_{t}^{N}f(x)}\leq\abs{S_{t}^{N}f(x)-S_{0}^{N}f(x)}+\abs{S_{0}^{N}f(x)}.
\end{equation}
Since $S_{0}^{N}f(x)=P_Nf(x)$, we have $\sup_N \abs{S_{0}^{N}f(x)}\leq  M_{dy}f(x)$.
Let us now estimate the first term on the right hand side of \eqref{eq:sumaryrestar}.
Let $Q(j,x)$ be the only cube $Q$ in $\D^j$ for which $x\in Q$.
Let $\mathscr{H}(j,x)$ be the set of all the wavelets based on $Q(j,x)$.
Recall that $\#(\mathscr{H}(j,x))$ is bounded by a purely geometric constant. Then
\begin{align*}
&\abs{S_{t}^{N}f(x)-S_{0}^{N}f(x)}\\
&= \abs{\sum_{j=0}^{N}\sum_{h\in\mathscr{H}(j,x)}\left[e^{-i t m_h\mu(Q(j,x))^{-\beta}}-1\right]\proin{f}{h}h(x)}\\
&= \abs{\sum_{j=0}^{N}\sum_{h\in\mathscr{H}(j,x)}\left[e^{-i t m_h\mu(Q(j,x))^{-\beta}}-1\right] \int_{Q(j,x)}(f(y)-f(x))h(y)
d\mu(y)\, h(x)}\\
&= C\, t \sum_{j=0}^{N}\sum_{h\in\mathscr{H}(j,x)}\frac{\abs{e^{-i t m_h\mu(Q(j,x))^{-\beta}}-1}}{m_h t \mu(Q(j,x))^{-\lambda}}
\frac{1}{\mu(Q(j,x))^{1+\lambda}}\int_{Q(j,x)}\abs{f(y)-f(x)}d\mu(y)\\
&\leq C\, t \sum_{j=0}^{N}\mu(Q(j,x))^{\lambda-\beta} M^{\#}_{\lambda, dy}f(x)\\
&\leq \widetilde{C}\, t M^{\#}_{\lambda, dy}f(x)
\end{align*}
and
\textit{(\ref{lem:estimacionesdosmaximales}.\ref{item:estimacionesdosmaximalesuno})}
is proved. Since $t<1$,
\textit{(\ref{lem:estimacionesdosmaximales}.\ref{item:estimacionesdosmaximalesdos})}
follows. And applying Lemma~\ref{lem:acotacionMaximalCalderon} we get
\textit{(\ref{lem:estimacionesdosmaximales}.\ref{item:estimacionesdosmaximalestres})}.
\end{proof}

\begin{proof}[Proof of Theorem~\ref{thm:Theorem3Main}]
Since for each $Q\in\D$ we have that $\mathcal{X}_Q$ belongs to $Lip(X,\delta)$,
in fact $\abs{\mathcal{X}_Q (x)-\mathcal{X}_Q (y)}\leq\frac{\delta (x,y)}{\mu(Q)}$,
then $\mathcal{S}(\mathscr{H})$ is a dense subspace of $L^2(X,\mu)$ such that
$S^N_t g(x)$ converges as $N\to \infty$ for every $x\in X$ and every $t\geq 0$.
From lemmas~\ref{lem:acotacionMaximalCalderon}, \ref{lem:estimacionesdosmaximales},
and the above remark the result of Theorem~\ref{thm:Theorem3Main} follows the
standard argument of pointwise convergence (see \cite{AiBoGo13}).
\end{proof}

Let us finally illustrate the above result in the Sierpinski triangle.
Set $S$ to denote the Sierpinski triangle in $\mathbb{R}^2$ equipped with the
normalized Hausdorff measure $\mathcal H_s$ of order $s=\tfrac{\log 3}{\log 2}$. For each positive integer $j$, the set $S$ can be written as the union
of $3^j$ translations of the contraction of $S$ by a factor $3^{-j}$, with
respect to any
one of the three vertices of the convex hull of $S$. Except for sets of $\mathcal H_s$
measure zero this covering of $S$ is disjoint. Each one of these disjoint
pieces is denoted by $T^j_k$, $k=1,\ldots,3^j$. We shall also write $T^0_1$ to denote $S$.
Set $\D^j$ to denote the family
$\{T^j_k: k=1,\ldots,3^j\}$ and $\D=\cup_{j\geq 0}\D^j$.
Notice that $\mathcal H_s(T^j_k)=3^{-j}$.

For each $j$ and each $k=1,\ldots,3^j$ we have that $T^j_k$ contains and is covered
by exactly three pieces $T^{j+1}_{k_1}$, $T^{j+1}_{k_2}$, $T^{j+1}_{k_3}$ of the
generation $j+1$. In order to abbreviate the notation, given $T\in\D$ we
write $T(1)$, $T(2)$ and $T(3)$ to denote these pieces of $T$. The three dimensional
space of all functions defined in $T$ which are constant on $T(1)$, $T(2)$ and
$T(3)$ has as a basis
$\set{\mathcal{X}_{T},\mathcal{X}_{T(1)},\mathcal{X}_{T(2)}}$. By
orthonormalization of this basis with the inner product of $L^2(S,d\mathcal H_s)$ keeping
$3^{j/2}\mathcal{X}_{T}$, $T\in\D^j$, we get two other functions $h^1_T$ and $h^2_T$ such that
$\set{3^{j/2}\mathcal{X}_{T},h^1_T,h^2_T}$
is an orthonormal basis for that 3-dimensional space. From the self similarity of our setting, we can take
$h^1_T$ and $h^2_T$ as the corresponding scalings and translations of those
associated to $T^0_1$. If
\begin{align*}
h^1_{T^0_1}(x)& = \sqrt{\tfrac{3}{2}}\left(\mathcal{X}_{T(1)}-\mathcal{X}_{T(2)}\right), \\
h^2_{T^0_1}(x) &= \tfrac{1}{\sqrt{2}}
\left(\mathcal{X}_{T(1)}+\mathcal{X}_{T(2)}-2\mathcal{X}_{T(3)}\right),
\end{align*}
then
\begin{equation*}
h^1_{T^j_k}(x)= 3^{\tfrac{j}{2}}\sqrt{\tfrac{3}{2}}\left(\mathcal{X}_{T^j_k(1)}-\mathcal{X}_{T^j_k(2)}\right)
\end{equation*}
and
\begin{equation*}
h^2_{T^j_k}(x)= 3^{\tfrac{j}{2}}\tfrac{1}{\sqrt{2}}\left(\mathcal{X}_{T^j_k(1)}+\mathcal{X}_{T^j_k(2)}-2\mathcal{X}_{T^j_k(3)}\right).
\end{equation*}

The system $\mathscr{H}$ of all these functions $h$ is an orthonormal basis for the subspace of $L^2(S,d\mathcal H_s)$ of those functions
with vanishing mean. Or, equivalently, the system $\mathscr{H}\cup \set{1}$ is an orthonormal basis for $L^2(S,d\mathcal H_s)$.

For a given $h$ in $\mathscr{H}$ we shall use $T(h)$ to denote the triangle
in which $h$ is based, i.e. $h=h^i_T$ for $i=1$ or $i=2$. For each $T\in\D$ there are exactly two wavelets $h\in\mathscr{H}$ based on $T$.

The function $\delta(x,y)=\min\{\mathcal H_s(T): T\in\D \textrm{ such that } x,y\in T\}$ is
a distance on $S$ provided we agree at defining $\delta(x,x)=0$.

Set $j(h)$ to denote the integer $j$ such that $T(h) \in \D^j$.
From Lemma~\ref{thm:eingefunctions_of_Dbeta} for $0<\beta<1$ we
have that each $h\in\mathscr{H}$ is an eigenfunction with eigenvalue
$m_\beta 3^{j(h)\beta}$ of the dyadic fractional differential operator
\begin{equation*}
D^\beta g(x)=\int_S\frac{g(x)-g(y)}{\delta(x,y)^{1+\beta}} d\mathcal H_s(y),
\end{equation*}
where the constant $m_\beta= 1+\frac{1}{2}\frac{1}{3^\beta-1}$ is independent
of $h$.

Our main results throughout this paper applied to the particular setting $S$ read as follows:
\begin{corollary}
Let $f$ be an $L^2(S,d\mathcal H_s)$ function with vanishing mean such that
\begin{equation*}
\iint_{S\times S}\frac{\abs{f(x)-f(y)}^2}{\delta(x,y)^{1+2\lambda}} d\mathcal H_s(x) d\mathcal H_s(y)<\infty
\end{equation*}
for some $\lambda>\beta>0$. Then, the wave function
\begin{equation*}
u(x,t)=\sum_{h\in\mathscr{H}}e^{-itm_\beta 3^{j(h)\beta}}\proin{f}{h} h(x),
\end{equation*}
solves the Schr\"{o}dinger type problem
\begin{equation*}
\begin{dcases}
i \frac{\partial u}{\partial t} = D^{\beta}u,  & \mbox{$x\in S$, $t>0$}\\
\medskip
\lim_{t\to 0^+}u(x,t)=f(x),  & \mbox{for $\mathcal H_s$-almost every $x\in S$.}
\end{dcases}
\end{equation*}
\end{corollary}


\def\cprime{$'$}
\providecommand{\bysame}{\leavevmode\hbox to3em{\hrulefill}\thinspace}
\providecommand{\MR}{\relax\ifhmode\unskip\space\fi MR }
\providecommand{\MRhref}[2]{%
  \href{http://www.ams.org/mathscinet-getitem?mr=#1}{#2}
}
\providecommand{\href}[2]{#2}


\bigskip

\bigskip

Instituto de Matem\'atica Aplicada del Litoral

CCT CONICET Santa Fe, Predio ``Dr. Alberto Cassano''

Colectora Ruta Nac. 168 km 0, Paraje El Pozo

Santa Fe, Argentina.

\bigskip

Marcelo Actis (\verb"mactis@santafe-conicet.gov.ar");

Hugo Aimar (\verb"haimar@santafe-conicet.gov.ar");

Bruno Bongioanni (\verb"bbongio@santafe-conicet.gov.ar");

Ivana G\'{o}mez (\verb"ivanagomez@santafe-conicet.gov.ar").

\end{document}